\documentclass[12pt, a4paper]{article}

\usepackage{amssymb,amsmath, amsthm}

\theoremstyle{plain}
\newtheorem{ex}{Example}[section]
\newtheorem{rem}{Remark}[section]

\newtheorem{thm}{Theorem}[section]
\newtheorem{lem}{Lemma}[section]

\newcommand{\E}{\mathbf{E}}

\title{ {\bf A Multiplicative Wavelet-based\\ Model  for Simulation\\  of a Random Process} }
\date{}

\begin{document}
\begin{center}
{\Large  \bf A Multiplicative Wavelet-based Model  for Simulation  of a Random Process}
\medskip

IEVGEN TURCHYN 

{\small 
University of Lausanne,  Lausanne, 
Switzerland 
} 

\end{center}






\noindent
{\it \small 
We consider a random process $Y(t)=\exp\{X(t)\}$, where $X(t)$  
is a centered second-order  process which correlation function 
$R(t,s)$ can be represented as \linebreak $\int_{\mathbb{R}} u(t,y)\overline{u(s,y)} dy.$
A multiplicative wavelet-based representation is found for $Y(t)$.  
We propose a model for simulation of the process $Y(t)$ 
and find its rates of convergence to the process in the spaces $C([0,T])$ 
and $L_p([0,T])$ for the case when 
$X(t)$ is a strictly sub-Gaussian process.  }






\begin{center}
{\bf Keywords} \ \    Sub-Gaussian random processes;      Simulation 

{\bf Mathematics Subject Classification} \ \  Primary 60G12 
\end{center}

\section{Introduction}
\label{sect-intro}


  Simulation of random processes is  a wide area nowadays,   
there exist many methods for simulation of stochastic processes (see e.g. \cite{Ogorodn, Ripley}).    

  But there exists one substantial problem:
  for most of traditional methods of simulation of random processes 
  it is difficult to measure the quality of approximation 
  of a process by its model in terms of ``distance" between paths of the process 
  and the corresponding paths of the model. 
  Therefore models for which such distance can be estimated are quite interesting.  
  
   There exists a concept for simulation  by such models which  
  is called si\-mu\-lation with given accuracy and reliability. 
%
Simulation with given accuracy and reliability is considered, for example, in \cite{Koz-Pogorilyak, Koz-Sottinen}.

  Simulation with given accuracy and reliability can be described in the following way. 
 An approximation $\hat{X}(t)$ of a random process $X(t)$ is built. 
The random process  $\hat{X}(t)$ is called a model of $X(t)$. 
A model depends on certain parameters.  
The rate of convergence of a model to a process is given by a statement of the following type: if numbers  $\delta$ (accuracy) and  $\varepsilon$ ($1-\varepsilon$ is called reliability) are given and the parameters of the model satisfy certain restrictions (for instance, they are not less than certain lower bounds) then 
\begin{equation}
\label{defi-acc-rel}
P\{\|X- \hat X \| >\delta \} \leq \varepsilon.  
\end{equation} 
Many such results have been proved for the cases when the norm in (\ref{defi-acc-rel}) is the $L_p$ norm or the uniform norm.  
But simulation with given accuracy and reliability has been developed so far almost only 
for processes which one-dimensional distributions have tails which  are 
not heavier than Gaussian tails (e.g. for sub-Gaussian processes).


 We consider a random process
$Y(t)=\exp\{X(t)\}$ and a 
scaling function~$\phi(x)$
with the corresponding wavelet $\psi(x)$, where $X(t)$ is a centered
second-order process such that its correlation function $R(t,s)$
can be represented as       
$$
R(t,s)=\int_{\mathbb{R}}u(t, \lambda)
\overline{u(s,\lambda)}d\lambda.
$$ 
We prove that 
$$
Y(t)=\prod\limits_{k \in \mathbb{Z}}
\exp\{\xi_{0k}a_{0k}(t) \}
\prod\limits_{ j=0}^{\infty} \prod\limits_{l \in \mathbb{Z}}
\exp\{\eta_{jl} b_{jl}(t) \},
$$
where $\xi_{0k}, \eta_{jl}$ are random variables, 
$a_{0k}(t), b_{jl}(t)$ are functions that depend on $X(t)$ and the wavelet.  

We take as a model of $Y(t)$ the process
$$
  \hat Y(t) 
  =
\prod \limits_{k=-(N_0-1)}^{N_0-1} 
\exp\{ \xi_{0k}a_{0k}(t)\} 
\prod\limits_{j=0}^{N-1}\prod\limits_{l=-(M_j-1)}^{M_j-1} \exp\{ \eta_{jl}b_{jl}(t)\}.  
$$

Let us consider the case when $X(t)$ is a sub-Gaussian process. 
Note that the class of processes $Y(t)=\exp\{X(t)\}$, where
$X(t)$ is a sub-Gaussian process,   
is a rich class which 
includes many processes which one-dimensional distributions 
have tails heavier than Gaussian tails, e.g. when $X(t)$ is a Gaussian process 
the one-dimensional distributions of $Y(t)$ are lognormal.  

We describe the rate of convergence
of $\hat Y(t)$ to a sub-Gaussian process $Y(t)$ in $C([0,T])$  in such a  way: 
if $\varepsilon \in (0; 1)$ and $\delta>0$ are given and the parameters $N_0, N, M_j $ are big enough then
\begin{equation}
\label{T-N-abstr}
\mathrm{P}\left\{\sup_{t \in [0,T]} |Y(t) / \hat{Y}(t) -1|>  \delta  \right\}
\leq \varepsilon.   
\end{equation}  
A similar statement which characterizes the rate of convergence
of $\hat Y(t)$ to~$Y(t)$ in $L_p([0,T])$ is also proved 
for the case when (\ref{T-N-abstr}) is  
replaced by the inequality 
$$
\mathrm{P}\left\{\left(\int_{0}^T |Y(t) - \hat{Y}(t)|^p dt\right)^{1/p}>  \delta  \right\}
\leq \varepsilon.   
$$  


If the process $X(t)=\ln Y(t)$  is Gaussian then   the model $\hat{Y}(t)$ can be used
 for computer simulation  of $Y(t)$. 

  One of the merits of our model is its simplicity. 
Besides, it can be used for simulation of processes which one-dimensional 
distributions have tails which are heavier than Gaussian tails.


\section{Auxiliary facts}


 A random variable $\xi$ is called {\it sub-Gaussian} if 
there exists such a constant $a \geq 0$ that 
$$
\mathsf{E} \exp\{\lambda \xi \} \leq \exp\{ \lambda^2 a^2 / 2 \}
$$
for all $\lambda \in \mathbb{R}$.  
 
The class of all sub-Gaussian random variables on a standard probability space 
$\{\Omega, \mathcal{B}, P\}$ is a Banach space with respect to the norm 
$$
\tau(\xi)=
\inf\{ a \geq 0:  
\mathsf{E}\exp\{\lambda \xi \} \leq \exp\{  \lambda^2 a^2/ 2 \} , \lambda \in \mathbb{R} \}.  
$$

 A centered Gaussian random variable and a random variable uniformly distributed on $[-b,b]$  are examples of sub-Gaussian random variables.

A sub-Gaussian random variable $\xi$ is  called  {\it strictly sub-Gaussian} if
$$
\tau(\xi) = (\mathsf {E} \xi^2)^{1/2}. 
$$ 

    For any sub-Gaussian random variable  $\xi$ 
\begin{equation}
\label{oc-expon-mom}
\mathsf{E}\exp\{\lambda \xi \} \leq \exp\{  \lambda^2 \tau^2(\xi)/ 2 \} , \  \lambda \in \mathbb{R},   
\end{equation}  
and 
\begin{equation}
\label{ots-mom-subg}
\mathsf{E} |\xi|^p \leq 2 \left(\frac{p}{e} \right)^{p/2} (\tau(\xi))^p, \ p>0. 
\end{equation}

A family $\Delta$ of sub-Gaussian random variables 
is called {\it strictly  sub-Gaussi\-an} if for any finite or countable  set $I$ of random variables $\xi_i \in \Delta$ and for any 
$\lambda_i \in \mathbb{R}$ 
$$
\tau^2 \left( \sum\limits_{i \in I} \lambda_i \xi_i  \right) =
 \mathsf{E} \left( \sum\limits_{i \in I} \lambda_i \xi_i \right)^2. 
$$
A stochastic process $X=\{X(t), t \in \mathbf{T} \}$ 
is called  {\it sub-Gaussian} if all the random variables 
$X(t), t \in \mathbf{T},$ 
are  sub-Gaussian.
We call a stochastic process \linebreak  $X  =\{X(t), t \in \mathbf{T} \}$ 
 {\it strictly sub-Gaussian} if the family
$\{X(t),  t \in \mathbf{T} \}$ 
is   strictly sub-Gaussian. Any centered Gaussian process is strictly sub-Gaussian.

%
%

Details about sub-Gaussian random variables and processes can be found in 
\cite{Buld-Koz-MetrChar}.

We will use wavelets (see \cite{Hern-Weiss} for details) for an expansion of a stochastic
process.  Namely, we  use a scaling function $\phi(x)$ of an MRA 
and the corresponding wavelet $\psi(x)$.  
Set  
$$
\phi_{0k}(x)= \phi(x - k), \quad   k \in \mathbb{Z}, 
$$ 
$$
\psi_{jl}(x)= 2^{j/2}\psi(2^{j}x - l), \quad  j,l \in \mathbb{Z}. 
$$     
We require orthonormality of the system $\{\phi(\cdot - k), \,  k \in \mathbb{Z} \}$. 
We denote by $\hat{f}$ the Fourier transform of a function $f \in L_2(\mathbb{R})$.

%

  The following statement is crucial for us.

\begin{thm}
\label{teor-rozkl1}
 (\cite{Koz-Tur-IJSMS})
 Let $X=\{X(t), t \in \mathbb{R}\}$ be centered
random process such that for all $t \in \mathbb{R}$  
$\,\,\, \mathsf{E} |X(t)|^2<\infty$. Let $R(t,s)=
\mathsf{E} X(t)\overline{X(s)}$ and there exists such a
Borel function $u(t,\lambda),$   $t \in \mathbb{R},
\lambda \in \mathbb{R}$ that 
\[
\int_{\mathbb{R}}|u(t,\lambda)|^2 d\lambda< \infty \quad
\text{for all} \quad t \in \mathbb{R}
\]
 and
\[
R(t,s)=\int_{\mathbb{R}}u(t, \lambda)
\overline{u(s,\lambda)}d\lambda.
\]
Let $\phi(x)$ be a scaling function, 
$\psi(x)$ --- the corresponding 
wavelet.   
 Then the process $X(t)$ can be presented as the  following series
 which converges for any $t \in \mathbb{R}$ in $L_2(\Omega)$: 
\begin{equation}
\label{rozkl1}
X(t)=\sum\limits_{k \in \mathbb{Z}}\xi_{0k}a_{0k}(t)+
\sum\limits_{j=0}^{\infty}\sum\limits_{l \in \mathbb{Z}}\eta_{jl}b_{jl}(t),
\end{equation}
where
\begin{equation}
\label{a-0k}
a_{0k}(t)=\frac{1}{\sqrt{2\pi}}
 \int_{\mathbb{R}} u(t,y)\overline{\hat \phi_{0k}(y)}dy=
\frac{1}{\sqrt{2\pi}}
 \int_{\mathbb{R}} u(t,y)\overline{\hat \phi(y)}e^{iyk}dy,
\end{equation}
\begin{equation}
\label{b-jk}
b_{jl}(t)=\frac{1}{\sqrt{2\pi}}
 \int_{\mathbb{R}} u(t,y)\overline{\hat \psi_{jl}(y)}dy=
\frac{1}{\sqrt{2\pi}}
 \int_{\mathbb{R}} u(t,y)2^{-j/2}\exp\Bigl\{\Bigr. i\frac{y}{2^j}l
 \Bigl\}\Bigr.
 \overline{\hat \psi \left(\frac{y}{2^j} \right)}dy,
\end{equation}
$\xi_{0k}, \eta_{jl}$ are centered random variables such that
\[
\mathsf{E} \xi_{0k}\overline{\xi_{0l}}=\delta_{kl}, \ 
\mathsf{E} \eta_{mk}\overline{\eta_{nl}}=\delta_{mn}\delta_{kl}, \ 
\mathsf{E} \xi_{0k}\overline{\eta_{nl}}=0 \, .
\]
\end{thm}

\noindent{\bf Definition.}
    Condition RC holds for stochastic process  
$X(t)$ if it satisfies the conditions of Theorem \ref{teor-rozkl1},
$u(t, \cdot) \in L_1(\mathbb{R})\cap L_2(\mathbb{R})$  
and inverse Fourier transform $\tilde{u}_{x}(t,x)$ of function 
$u(t,x)$ with respect to $x$ is a real function.    

\begin{rem}
Condition RC guarantees that
the coefficients $a_{0k}(t)$, $b_{jl}(t)$ of expansion $(\ref{rozkl1})$ are real.    
\end{rem}

Suppose that  $X(t)$ is a process which satisfies the conditions of 
Theorem~\ref{teor-rozkl1}.
Let us consider the following approximation (or model) of $X(t)$:  
$$
\hat X(t)=
\hat X(N_0, N, M_0, \ldots, M_{N-1}, t)
$$
\begin{equation}
\label{model-Xhat}
= 
\sum\limits_{k=-(N_0-1)}^{N_0-1} \xi_{0k}a_{0k}(t)+
\sum\limits_{j=0}^{N-1}\sum\limits_{l=-(M_j-1)}^{M_j-1}\eta_{jl}b_{jl}(t),  
\end{equation}
where  $\xi_{0k}, \eta_{jl}, a_{0k}(t),  b_{jl}(t)$ are defined in 
Theorem~\ref{teor-rozkl1}.   

Approximation of Gaussian and sub-Gaussian processes  by model (\ref{model-Xhat})  
has been studied in \cite{Koz-Tur-IJSMS} and  \cite{Turchyn-MCMA}. 

\begin{rem}
 If $X(t)$ is a Gaussian process then we can take as 
$\xi_{0k}, \eta_{jl}$ in~$(\ref{model-Xhat})$ independent random variables with  
distribution $N(0;1)$. 
\end{rem}

\section{A multiplicative representation}
  
   We will obtain a multiplicative representation for a wide class of stochastic processes.  
    
\begin{thm}
\label{thm-product-repr}
Suppose that a random process $Y(t)$
can be represented as  $Y(t)=\exp\{X(t)\}$, where the process $X(t)$ satisfies the conditions of Theorem~\ref{teor-rozkl1}. Then 
the equality 
\begin{equation}
\label{prod-expn}
Y(t)=\prod\limits_{k \in \mathbb{Z}}
\exp\{\xi_{0k}a_{0k}(t) \}
\prod\limits_{ j=0}^{\infty} \prod\limits_{l \in \mathbb{Z}}
\exp\{\eta_{jl} b_{jl}(t) \}
\end{equation}
holds, 
where product $(\ref{prod-expn})$ converges in probability for any fixed $t$ and 
$\xi_{0k}, \eta_{jl}, \linebreak  a_{0k}(t),  b_{jl}(t)$ are defined in  
Theorem~\ref{teor-rozkl1}. 
\end{thm}
The statement of the theorem immediately follows from Theorem~\ref{teor-rozkl1}.

\begin{rem}
   It was shown in \cite{Koz-Tur-IJSMS} that any  centered second-order wide-sense stationary process $X(t)$ which has the spectral density  
satisfies the conditions of Theorem~\ref{teor-rozkl1}.      
  The process $Y(t)=\exp\{X(t)\}$ can be represented as product (\ref{prod-expn}) and therefore the class of processes which satisfy the conditions of 
  Theorem~\ref{thm-product-repr} is wide enough.  
\end{rem}

 It is natural to approximate a stochastic process
$Y(t)=\exp\{X(t)\}$ 
which satisfies the conditions of
Theorem~\ref{thm-product-repr}  
by the  model       
$$
  \hat Y(t) 
  =
\hat Y(N_0, N, M_0, \ldots, M_{N-1}, t)  
$$
\begin{equation}
\label{model-Yhat}
=
\prod \limits_{k=-(N_0-1)}^{N_0-1} 
\exp\{ \xi_{0k}a_{0k}(t)\} 
\prod\limits_{j=0}^{N-1}\prod\limits_{l=-(M_j-1)}^{M_j-1} \exp\{ \eta_{jl}b_{jl}(t)\}
   = \exp\{\hat X(t)\}.  
\end{equation} 

\begin{rem}
 If $X(t)= \ln Y(t)$ is a Gaussian process then we can use the model 
$\hat Y(t)$ for computer simulation of $Y(t)$,   
taking as  $\xi_{0k}, \eta_{jl}$ in (\ref{model-Yhat}) independent random variables with  
distribution $N(0;1)$. 
\end{rem}

\section{Simulation with given relative  accuracy  \\ and reliability  in $C([0,T])$}

Let us study the rate of convergence in $C([0,T])$ of 
 model (\ref{model-Yhat})  to a process~$Y(t)$. 
We will need several auxiliary facts.

\begin{lem} 
(\cite{Turchyn-MCMA}) 
\label{lem-ocen-koef}
  Let $X=\{X(t), t \in \mathbb{R} \}$ be a centered stochastic process which satisfies  the requirements of Theorem \ref{teor-rozkl1},  
$T>0$,  $\phi$ be a  sca\-ling function, $\psi$ 
 be the corresponding wavelet, the function $\hat\phi(y)$ be absolutely continuous on any interval,
the function $u(t,y)$ be absolutely continuous with respect to $y$
for any fixed~$t$, there exist the derivatives
$u_{\lambda}'(t,\lambda),\hat{\phi}'(y), \hat{\psi}'(y)$ 
and  
$|\hat{\psi}'(y)| \leq C,\,\,\, 
|u(t, \lambda)| \leq |t| u_1(\lambda), 
|u_{\lambda}'(t,\lambda)| \leq  |t| \, u_2(\lambda),$
\begin{equation}
\label{u1-conds}
\int_{\mathbb{R}} u_1(y) |y| dy < \infty, \quad
\int_{\mathbb{R}} u_1(y) dy < \infty, \quad
\int_{\mathbb{R}} u_1(y) |\hat{\phi}'(y)|dy < \infty,    
\end{equation}
\begin{equation}
\label{u1-u2-conds}
\int_{\mathbb{R}} u_1(y) |\hat{\phi}(y)| dy< \infty,
\quad
\int_{\mathbb{R}} u_2(y) |y| dy < \infty, \quad
\int_{\mathbb{R}} u_2(y) |\hat{\phi}(y)| dy < \infty,
\end{equation}
\[
\lim \limits_{|y| \rightarrow \infty} 
u(t,y) \, \overline{\hat{\psi}(y/2^j)}=0 
\quad  \forall j=0,1,\ldots 
\quad \forall t \in [0,T]
\]
and
\[
\lim \limits_{|y| \rightarrow \infty} 
u(t,y) \hat{\phi}(y)=0 \quad \forall t \in [0,T],
\]
\[
 E_1=\frac{\textstyle 1}{\textstyle \sqrt{2\pi} }	
\int_{\mathbb{R}} u_1(y) |\widehat{\phi}(y)| dy, 
\]
\[
 E_2=\frac{\textstyle 1}{\textstyle \sqrt{2\pi}}
\left( \int_{\mathbb{R}} u_1(y) |\widehat{\phi}'(y)| dy +
\int_{\mathbb{R}} u_2(y) |\widehat{\phi}(y)| dy \right),
\]
$$
F_1=\frac{\textstyle C}{\textstyle \sqrt{2\pi}} 
\int_{\mathbb{R}} u_1(y) |y| dy, 
$$
$$
F_2=\frac{\textstyle C}{\textstyle \sqrt{2\pi}}
 \int_{\mathbb{R}} (u_1(y)+|y|u_2(y))dy.
$$
Let the process $\hat{X}(t)$ be defined by 
$(\ref{model-Xhat})$,  
$\delta>0$. If 
$N_0,\, N,\,M_j  \,  (j=0,1,  \ldots 
\linebreak \ldots, N-1)$ satisfy the  inequalities 
\[
N_0>\frac{6}{\delta}E_2^2 T^2 +1,
\]
\[
N>\max\left\{ 1+\log_{2}\left(\frac{72 F_2^2 T^2}{5 \delta}\right),
1+\log_{8}\left(\frac{18 F_1^2 T^2}{7 \delta}\right) \right\},
\]
\[
M_j>1+\frac{12}{\delta}F_2^2 T^2, 
\]
then
\begin{equation}
\label{Otkl-ot-mod}
\sup\limits_{t \in [0,T] } \mathsf{E}|X(t)-\widehat{X}(t)|^2 \leq \delta.
\end{equation}

\end{lem}

%
%

\begin{lem}
\label{lem-ocen-mod-nepr}
(\cite{Turchyn-MCMA})
  Let $X=\{X(t), t \in \mathbb{R} \}$ be a centered stochastic process which satisfies the requirements of Theorem \ref{teor-rozkl1}, 
$T>0$, $\phi$ be a sca\-ling function, $\psi$ be the correspon\-ding
 wavelet, $S(y)=\overline{\hat\psi(y)},\,\, 
S_{\phi}(y)=\overline{\hat\phi(y)}$; 
$\phi(y),  u(t, \lambda), S(y),  S_{\phi}(y)$ satisfy such conditions: the function $u(t,y)$ is absolutely continuous with
respect to $y$, the function $\hat\phi(y)$ is absolutely continuous,
\[
 |S'(y)|\leq M <\infty,
\] 
\[
\lim \limits_{|y| \rightarrow \infty} 
u(t,y) S(y/2^j)=0, 
\quad j=0,1,\ldots, \quad t \in [0,T],
\]
\[
\lim \limits_{|y| \rightarrow \infty} 
u(t,y) S_{\phi}(y)=0, \quad  t \in [0,T] ,
\]
there exist functions $v(y)$ and $w(y)$ such that
\[
|u'_{y}(t_1,y)-u'_{y}(t_2,y)|\leq |t_2-t_1|v(y),
\]
\[
|u(t_1,y)-u(t_2,y)|\leq |t_2-t_1| w(y)
\]  
and  
\[
\int_{\mathbb{R}} |y|v(y)dy< \infty , \quad
 \int_{\mathbb{R}} v(y) |S_{\phi}(y)|dy< \infty ,
\]
\[
\int_{\mathbb{R}} w(y)|S'_{\phi}(y)|dy< \infty,
\quad
\int_{\mathbb{R}}w(y)dy< \infty, 
\]
\[
\int_{\mathbb{R}}w(y)|y|dy< \infty, \quad
\int_{\mathbb{R}} w(y)|S_{\phi}(y)|dy< \infty;
\]     
$a_{0k}(t)$ and $b_{jl}(t)$ are defined by equalities   
$(\ref{a-0k})$ and $(\ref{b-jk})$,
\[
A^{(1)}=\frac{1}{\sqrt{2\pi}} \left( \int_{\mathbb{R}}
 v(y)|S_{\phi}(y)|dy + 
 \int_{\mathbb{R}} w(y)|S'_{\phi}(y)|dy \right),
\]
\[
B^{(0)}=\frac{M}{\sqrt{2\pi}} \int_{\mathbb{R}}w(y)|y|dy,
\]
\[
B^{(1)}=\frac{M}{\sqrt{2\pi}}
 \int_{\mathbb{R}}(w(y)+|y|v(y))dy,
\]
\[
C_{\Delta X}
=\sqrt{\frac{2(A^{(1)})^2}{N_0-1} +
\frac{(B^{(0)})^2}{7\cdot 8^{N-1}} + 
\frac{(B^{(1)})^2}{2^{N-3}} + 
(B^{(1)})^2 \sum\limits_{j=0}^{N-1}\frac{1}{2^{j-1}(M_j-1)}
} \ .
\]

  Then for   $t_1, t_2 \in [0,T]$ and  $N>1, N_0>1,M_j>1$ the inequality
\[
\sum\limits_{|k|\geq N_0}|a_{0k}(t_1)-a_{0k}(t_2)|^2 +
\sum\limits_{j \geq N}\sum\limits_{l \in \mathbb{Z}} |b_{jl}(t_1)-b_{jl}(t_2)|^2 
\]
\begin{equation}
\label{ner-ocen-mod-nepr}
+\sum\limits_{j=0}^{N-1}\sum\limits_{|l|\geq M_j}|b_{jl}(t_1)-b_{jl}(t_2)|^2
\leq C_{\Delta X}^2 (t_2-t_1)^2
\end{equation}
holds.
\end{lem}

\begin{lem}
\label{lem-ocen-C-DeltaX}

  If 
$$
  N_0 \geq 1 + \frac{8 (A^{(1)})^2}{\varepsilon^2}, 
$$
$$
  N \geq \max\left\{1+\log_8\frac{4 (B^{(0)})^2}{7 \varepsilon^2}, \ 
               3+\log_2\frac{4 (B^{(1)})^2}{\varepsilon^2}   \right\}, 
$$
$$
  M_j \geq 1+ 16 \frac{(B^{(1)})^2}{\varepsilon^2}     
$$
then 
$$
C_{\Delta X} \leq \varepsilon,  
$$ 
where  $A^{(1)}, B^{(0)}, B^{(1)}, C_{\Delta X}$ are defined 
in Lemma~\ref{lem-ocen-mod-nepr}.  
\end{lem}
 We omit the proof due to its triviality.

\noindent{\bf Definition.} 
  We say that a  model $\hat Y(t)$
approximates a stochastic process~$Y(t)$ with given {\it relative accuracy} $\delta$
and {\it reliability} $1-\varepsilon$ (where $\varepsilon \in (0; 1)$)  in $C([0,T])$ if 
$$
\mathrm{P}\left\{\sup_{t \in [0,T]} |Y(t) / \hat{Y}(t) -1|>  \delta  \right\}
\leq \varepsilon.   
$$

Now we can formulate a result on the 
rate of convergence in  $C([0,T])$. 

\begin{thm}
\label{Main-thm}
  Suppose that a random process 
  $Y=\{Y(t),  t \in \mathbb{R}\}$
  can be represented as  $Y(t)=\exp\{X(t)\}$, where a separable strictly sub-Gaussian random process 
$X=\{X(t),  t \in \mathbb{R}\}$ is mean square continuous, satisfies the condition RC and the conditions of
Lemmas \ref{lem-ocen-koef} and \ref{lem-ocen-mod-nepr} together with a 
\linebreak  scaling function $\phi$ and the corresponding  wavelet $\psi$, the random variables
 $\xi_{0k}, \eta_{jl}$ in expansion $(\ref{rozkl1})$ of the process $X(t)$ are independent strictly sub-Gaussian, $\hat X(t)$ is a model of  $X(t)$ defined by $(\ref{model-Xhat})$,    
 $\hat Y(t)$  is defined by $(\ref{model-Yhat})$,  
 $\theta  \in (0;1)$, $\delta  >0$, $\varepsilon \in (0;1)$, $T>0$, the  numbers $A^{(1)}, B^{(0)}, B^{(1)}, E_2, F_1, F_2$ are defined in Lemmas \ref{lem-ocen-koef} and
\ref{lem-ocen-mod-nepr},  
$$
\hat \varepsilon= \delta \sqrt{\varepsilon}, 
$$
$$
A(\theta)= \int_{1/(2 \theta)}^{\infty} \frac{\sqrt{v+1}}{v^2}dv,   
$$
$$
\tau_1= \frac{e^{1/2}\,  \hat\varepsilon}
{  2^{7/4}(64 + \hat\varepsilon^2)^{1/4} }, 
$$
$$
\tau_2= (32 \ln(1+\hat\varepsilon^2/60))^{1/2}, 
$$
$$
\tau_3= \sqrt{\ln(1+\hat\varepsilon^3/8)} \Bigl /\Bigr.  \sqrt{2}, 
$$
$$
\tau_{*}= \min \{\tau_1,  \tau_2, \tau_3 \},   
$$
$$
Q=  
\frac{ e^{1/2}\hat\varepsilon \, \theta (1-\theta)  } 
 { 2^{9/4} A(\theta) T (1+ \hat\varepsilon^3/8) },      
$$
$$
  N_0^{*} = 1 + \frac{8 (A^{(1)})^2}{Q^2}, 
$$
$$
  N^* =  \max\left\{1+\log_8\frac{4 (B^{(0)})^2}{7 Q^2}, \ 
               3+\log_2\frac{4 (B^{(1)})^2}{Q^2}   \right\}, 
$$
$$  
M^* =  1+ 16 \frac{(B^{(1)})^2}{Q^2},       
$$
$$
N_0^{**} = \frac{6}{\tau_*^2}E_2^2 T^2 +1,
$$
$$
N^{**} = \max\left\{ 1+\log_{2}\left(\frac{72 F_2^2 T^2}{5 \tau_*^2}\right),
1+\log_{8}\left(\frac{18 F_1^2 T^2}{7 \tau_*^2}\right) \right\},
$$
$$
M^{**}  = 1+\frac{12}{\tau_*^2}F_2^2 T^2.  
$$

Suppose also that 
\begin{equation}
\label{SUP-TAU}
\sup\limits_{t \in [0,T]} \mathsf{E}(X(t)-\hat X(t))^2>0.
\end{equation}

  If 
\begin{equation}
\label{LB-N0}
N_0> \max\{N_0^{\ast}, N_0^{**} \}, 
\end{equation}
\begin{equation}
\label{LB-N}
N>\max\{N^{\ast}, N^{**} \},
\end{equation}
\begin{equation}
\label{LB-MJ}
M_j>\max\{M^{\ast}, M^{**} \} 
\quad  (j=0,1,\ldots,N-1),
\end{equation}
then the model $\hat Y(t)$ approximates 
the process $Y(t)$ with given relative accuracy $\delta$ and 
reliability $1-\varepsilon$ in $C([0,T])$. 
%
\end{thm}
\begin{proof}

  Denote
$$
\Delta X(t) = X(t) - \hat X(t), 
$$
$$
U(t)= Y(t)/\hat Y(t) - 1 = 
 \exp\{\Delta X(t)\} - 1, 
$$
$$
\rho_U(t,s) = \|U(t)-U(s) \|_{L_2(\Omega)}, 
$$
$$
\tau_{\Delta X}= \sup_{t \in [0,T]} \tau(\Delta X(t)). 
$$
Let us note that $\rho_U$ is a pseudometric. 
Let $N(u)$ be the metric massiveness of $[0,T]$ with respect to $\rho_U$, 
i.e. the minimum number of closed balls in the space $([0,T], \rho_U)$ with diameters at most $2u$ needed to cover~$[0,T]$, 
$$
\varepsilon_0=\sup_{t,s \in [0,T]}\rho_U(t,s).
$$ 
 We will denote the norm in $L_2(\Omega)$
 as $\|\cdot\|_2$ below.

Since $U(t) \in L_2(\Omega), t \in [0,T],$  we obtain using Theorem 3.3.3  from \cite{Buld-Koz-MetrChar} (see p.~98)
\begin{equation}
\label{EQ-sup-U}
\mathrm{P}\left\{ \sup_{t \in [0,T]} 
|U(t)|>\delta  \right\} \leq  \frac{S_2^2}{\delta^2}, 
\end{equation}   
where  
$$
S_2=  \sup_{t \in [0,T]}(\mathsf{E}|U(t)|^2)^{1/2}   + 
\frac{1}{\theta (1-\theta)} \int_{0}^{\theta \varepsilon_0} N^{1/2}(u)du. 
$$

 We will prove that 
 $S_2 \leq \delta \sqrt{\varepsilon} = \hat \varepsilon$.

  First of all let us estimate $\E|U(t)|^2$, where $t \in [0,T]$.  
  
 Using the inequality
\begin{equation}
\label{razn-exp}
|e^a - e^b| \leq |a-b| \max\{e^a, e^b \} \leq  |a-b| (e^a + e^b)
\end{equation}    
(we set $b=0$)
 and Cauchy-Schwarz inequality we obtain  
$$
\mathsf{E}|U(t)|^2 =\E (\exp\{\Delta X(t)\} - 1)^2 
\leq
(\E |\Delta X(t)|^4)^{1/2}   
(\E (\exp\{\Delta X(t)\} + 1)^4)^{1/2}.   
$$
  
It follows from  (\ref{ots-mom-subg}) that  
\begin{equation}
\label{E-Delta-Xv4}
\E |\Delta X(t)|^4 \leq \frac{32}{e^2}\, 
\tau_{\Delta X}^4.  
\end{equation}

Let us estimate 
$G = \E (\exp\{\Delta X(t)\} + 1)^4$. 
Since 
$$
\E \exp\{k \Delta X(t)\} 
\leq \exp\{k^2 \tau^2(\Delta X(t))/2 \} 
=A^{k^2}\leq A^{16}, \ 1\leq k \leq 4,
$$
where $A=\exp\{\tau_{\Delta X}^2 /2 \}$, 
we have 
\begin{equation}
\label{ineq-for-G}
G \leq \sum\limits_{k=1}^{4} 
{4 \choose k} A^{16} + 1 = 
15 A^{16} + 1. 
\end{equation}

  It follows from Lemma \ref{lem-ocen-koef}  and (\ref{LB-N0})--(\ref{LB-MJ})   that
\begin{equation}
\label{oc-tau-DeltaX}
\tau_{\Delta X} =
\sup_{t \in [0,T]} \mathsf{E}(|\Delta X(t)|^2)^{1/2} \leq \tau_{*}.
\end{equation}
Using (\ref{E-Delta-Xv4})--(\ref{oc-tau-DeltaX}) we obtain    
\begin{equation}
\label{nervo-E-U-sq}
(\mathsf{E}|U(t)|^2)^{1/2} 
\leq \hat \varepsilon/2.  
\end{equation}

  Let us estimate now 
$$   
I(\theta)=  \frac{1}{\theta (1-\theta)} \int_{0}^{\theta \varepsilon_0} N^{1/2}(u)du. 
$$

 At first we will find an upper bound for 
$N(u)$. In order to do this we will prove that 
\begin{equation}
\label{oc-razn-U}
\|U(t_1)- U(t_2) \|_{2} \leq 
C_U |t_1 - t_2| ,  
\end{equation}
where
$$
C_U = (2^{9/4}/e^{1/2}) 
C_{\Delta X} \exp\{2 \tau_{\Delta X}^2\}, 
$$ 
 $C_{\Delta X}$ is defined in Lemma~\ref{lem-ocen-mod-nepr}. 

We have, using (\ref{razn-exp})
and Cauchy-Schwarz inequality: 
$$
\|U(t_1)- U(t_2) \|_{2}^2 
= 
\E |\exp\{\Delta X(t_1)\} - \exp\{\Delta X(t_2) \} |^2  
$$
$$
\leq  \E |\Delta X(t_1) - \Delta X(t_2)|^2  
(\exp\{\Delta X(t_1)\} + \exp\{\Delta X(t_2) \} )^2
$$
$$
\leq 
(\E (\Delta X(t_1) - \Delta X(t_2))^4)^{1/2}   
(\E(\exp\{\Delta X(t_1)\} + \exp\{\Delta X(t_2)\})^4  )^{1/2}.  
$$

Applying (\ref{ots-mom-subg}), we obtain  
\begin{equation}
\label{oc-razn-v-4j}
(\E (\Delta X(t_1) - \Delta X(t_2))^4)^{1/2}   
\leq (2^{5/2}/e) C_{\Delta X}^2 |t_2 - t_1|^2. 
\end{equation}

Let us find an upper bound for 
$$
H= \E(\exp\{\Delta X(t_1)\} + \exp\{\Delta X(t_2)\})^4. 
$$

Since 
$$
\E\exp\{k \Delta X(t_1) + l \Delta X(t_2)\} 
$$
$$
\leq 
\exp\{\tau^2 ( k \Delta X(t_1) + l \Delta X(t_2) )/2 \}
\leq 
\exp\{ ( k \tau(\Delta X(t_1)) + l \tau(\Delta X(t_2)) )^2/2 \} 
$$
$$
\leq  \exp\{ 8 \tau_{\Delta X}^2\},
$$
where $k+l=4$, we have: 
\begin{equation}
\label{oc-H}
H \leq \sum_{k=0}^4 {4 \choose k} 
\exp\{ 8 \tau_{\Delta X}^2\} =
 16 \exp\{ 8 \tau_{\Delta X}^2\}   
\end{equation}
and (\ref{oc-razn-U}) follows from 
(\ref{oc-razn-v-4j}) and 
(\ref{oc-H}).

  Using inequality (\ref{oc-razn-U}),   	
 simple properties of metric entropy 
 (see \cite{Buld-Koz-MetrChar}, Lemma 3.2.1, p.~88) and the inequality 
 $$
 N_{\rho_1}(u) \leq T/(2u) + 1 
 $$
 (where $N_{\rho_1}$ is the entropy of $[0, T]$ with respect to the Euclidean  metric)   
  we have 
$$
N(u) \leq \frac{T C_U}{2u} + 1.  
$$   
Since   
$\varepsilon_0\leq C_U T$ we obtain      
$$
\int_{0}^{\theta \varepsilon_0}  
N^{1/2}(u) du \leq 
\int_{0}^{\theta \varepsilon_0}  
\left( T C_U/(2u) + 1  \right)^{1/2}
 du 
$$
\begin{equation}
\label{oc-entr-int}
= \frac{T C_U}{2} \int_{T C_U/(2 \theta \varepsilon_0)}^{\infty} 
\frac{\sqrt{v+1}}{v^2} dv   \leq  T C_{U} A(\theta)/2.  
\end{equation}

It is easy to check using 
Lemma \ref{lem-ocen-C-DeltaX}   
that under the conditions of the theorem  
the inequality  
\begin{equation}
\label{eq-oc-C}
C_{\Delta X} \leq Q
\end{equation}
holds. 
It follows from (\ref{oc-tau-DeltaX})  and 
(\ref{eq-oc-C})
that 
$$
C_U \leq \frac{\hat \varepsilon \,  \theta (1-\theta)}{T A(\theta)} 
$$
and therefore using (\ref{oc-entr-int})
we obtain 
\begin{equation} 
\label{nervo-I-theta}
I(\theta) \leq \hat \varepsilon/2. 
\end{equation}

  Now the statement of the theorem follows from 
(\ref{EQ-sup-U}), (\ref{nervo-E-U-sq}) and  (\ref{nervo-I-theta}).       
\end{proof}

\begin{ex} {\rm
  Let us consider a function 
$u(t,\lambda)=
t/(1+t^2+\lambda^2)^4$ 
and an arbitrary Dau\-be\-chies wavelet (with the corresponding 
scaling function $\phi$ and \linebreak the wavelet $\psi$). We will use the notations
\[
a_{0k}(t)=\frac{1}{\sqrt{2\pi}}\int_{\mathbb{R}}
u(t,y)\overline{\hat\phi_{0k}(y)}dy, \quad
b_{jl}(t)=\frac{1}{\sqrt{2\pi}}\int_{\mathbb{R}}
u(t,y)\overline{\hat\psi_{jl}(y)}dy
\]
and consider the stochastic process  
\[
X(t)=\sum_{k \in \mathbb{Z}}\xi_{0k}a_{0k}(t)+
\sum_{j=0}^{\infty}\sum_{l \in \mathbb{Z}}\eta_{jl}b_{jl}(t),
\]
where $\xi_{0k}, \eta_{jl} \,\,(k, l \in \mathbb{Z}, j=0, 1, \ldots)$ are 
independent   uniformly distributed over $[-\sqrt{3}, \sqrt{3}]$. 
It is easy to see that the process $Y(t)=\exp\{X(t)\}$ and the Daubechies wavelet satisfy the conditions of Theorem \ref{Main-thm}.
}
\end{ex}


\section{Simulation with given accuracy \\  and reliability in $L_p([0,T])$}

   Now we will consider the rate of convergence in $L_p([0,T])$ of model (\ref{model-Yhat}) to a
 process $Y(t)$.

%

\begin{lem}
\label{p6lem-ocen-koef-rozkl1-nestats} 
  Suppose that a centered stochastic process 
$X=\{X(t),  \linebreak t \in \mathbb{R} \}$  satisfies the conditions of  
 Theorem~\ref{teor-rozkl1}, $\phi$ is a  scaling function, $\psi$ is the corresponding  wavelet,
$\hat{\phi}$ and $\hat{\psi}$ are Fourier transforms of 
$\phi$ and $\psi$ respectively, $\hat\phi(y)$ is absolutely 
continuous, $u(t,y)$ is defined in  Theorem~\ref{teor-rozkl1} and $u(t,y)$ is absolutely 
continuous for any fixed $t$, there exist derivatives  
$u_{y}'(t,y), \hat{\phi}'(y),  \hat{\psi}'(y)$ and  
$|\hat\psi'(y)| \leq C$, 
$|u(t, y)|  \leq u_1(y),
|u_{y}'(t,y)| \leq |t| \, u_2(y), $
equalities $(\ref{u1-conds})$ and $(\ref{u1-u2-conds})$ hold, 
$$
\lim \limits_{|y| \rightarrow \infty} 
u(t,y) \, \overline{\hat{\psi}(y/2^j)}=0 
\,\,\,  \forall j=0,1,\ldots 
\,\,\, \forall t \in \mathbb{R},
$$
$$
\lim \limits_{|y| \rightarrow \infty} 
u(t,y) |\hat{\phi}(y)|=0 \,\,\, \forall t \in \mathbb{R} ;
$$
$$
S_1=\frac{1}{\sqrt{2\pi}}
 \int _{\mathbb{R}} u_1(y) |\hat{\phi}'(y)| d y, \quad 
S_2=\frac{1}{\sqrt{2\pi}} 
\int _{\mathbb{R}} u_2(y) |\hat{\phi}(y)| d y,
$$
$$
Q_1=\frac{C}{\sqrt{2\pi}} \int _{\mathbb{R}} u_1(y) d y, \quad
Q_2=\frac{C}{\sqrt{2\pi}} \int _{\mathbb{R}} u_2(y) |y| d y.
$$
  
  Then the following inequalities hold for the coefficients 
$a_{0k}(t), b_{jl}(t)$ in expansion $(\ref{rozkl1})$  
of the process $X(t)$: 
\begin{equation}
\label{9}
|a_{00}(t)| \leq \frac{1}{ \sqrt{2\pi} }	
\int _{\mathbb{R}} u_1(y) |\hat{\phi}(y)| d y,
\end{equation} 
\begin{equation}
\label{8}
|b_{j0}(t)| \leq \frac{C}{ \sqrt{2\pi} \, 2^{3j/2} }	
\int _{\mathbb{R}} u_1(y) |y| d y, \quad j=0,1, \ldots , 
\end{equation}
\begin{equation}
\label{6}
  |a_{0k}(t)| \leq \frac{S_1+S_2 |t|}{|k|}, \quad k \neq 0,	
\end{equation}
\begin{equation}
\label{7}
|b_{jl}(t)| \leq \frac{Q_1+Q_2 |t|}{2^{j/2} |k|},	
\quad k \neq 0, \quad j=0,1, \ldots 	
\end{equation}
\end{lem} 

  The proof of inequalities (\ref{9})--(\ref{7}) is 
  analogous to the proof of similar inequalities for the coefficients of 
expansion (\ref{rozkl1}) of a stationary process in \cite{Koz-Tur-IJSMS}.     
   

\begin{lem}
\label{nas-TN-subg-nestats-rozk1}
  Suppose that a  random process 
 $X=\{X(t),  t \in \mathbb{R} \}$ satisfies the conditions of  
  Theorem~\ref{teor-rozkl1}; a scaling function $\phi$ and the corresponding wave\-let~$\psi$ 
together with the process $X(t)$ satisfy the conditions of 
Lemma~\ref{p6lem-ocen-koef-rozkl1-nestats}, 
$C, Q_1, Q_2, S_1, S_2,  u_1(y)$  are defined in 
Lemma~\ref{p6lem-ocen-koef-rozkl1-nestats}, $T>0$, $p \geq 1,$ $\delta \in (0;1)$,  $\varepsilon>0$, 
$$
\delta_1=\min\left\{\frac{  \varepsilon^2 }
{  2T^{2/p} \ln(2/\delta)},\,\,\frac{  \varepsilon^2 }
{  p T^{2/p}} \right\}, \quad 
D=\frac{\textstyle{C}}{\textstyle{ \sqrt{2\pi} }} 
\int_{\mathbb{R}} u_{1}(y)|y|d y.
$$

   If 
$$
N_0> \frac{6}{\delta_1}(S_1+S_2 T)^2 + 1,
$$
$$
N>\max \left\{ 1 + \log_2 \left(\frac{72(Q_1+Q_2 T)^2}
{5 \delta_1} \right),
1 + \log_8 \left( \frac{18D^2}{7 \delta_1} \right) \right \},
$$
$$
M_j> 1 + \frac{12}{\delta_1}(Q_1+Q_2 T)^2  \left( 1-\frac{1}{2^{N}} \right),
$$
then 
$$
 \sup\limits_{t \in [0,T]} 
 \mathsf{E}|X(t)-\widehat{ X}(t)|^2\leq \delta_1. 
$$

%
\begin{proof}
We have 
$$ 
 \mathsf{E}|X(t)-\widehat{ X}(t)|^2 = 
 \sum_{k: |k|\geq N_0} |a_{0k}(t)|^2  +
\sum_{j=0}^{N-1} \sum_{l: |l|\geq M_j} |b_{jl}(t)|^2 +
\sum_{j=N}^{\infty} \sum_{l \in \mathbb{Z}} |b_{jl}(t)|^2.  
$$
It remains to apply inequalities (\ref{9})--(\ref{7}). 
\end{proof}
\end{lem}


\noindent{\bf Definition.} 
  We say that a  model $\hat Y(t)$
approximates a stochastic process~$Y(t)$ with given {\it accuracy} $\delta$
and {\it reliability} $1-\varepsilon$ (where $\varepsilon \in (0; 1)$)  in $L_p([0,T])$ if 
$$
\mathrm{P}\left\{\left(\int_{0}^T |Y(t) - \hat{Y}(t)|^p dt\right)^{1/p}>  \delta  \right\}
\leq \varepsilon.   
$$

\begin{thm}
\label{thm-T-N-v-Lp}
  Suppose that a random process 
  $Y=\{Y(t),  t \in \mathbb{R}\}$
  can be represented as  $Y(t)=\exp\{X(t)\}$, where a separable strictly sub-Gaussian random process 
$X=\{X(t),  t \in \mathbb{R}\}$ is mean square continuous, satisfies the condition RC and the conditions of
Lemma~\ref{nas-TN-subg-nestats-rozk1}  together with a  
  scaling function~$\phi$ and the corresponding  wave\-let $\psi$, the random variables
 $\xi_{0k}, \eta_{jl}$ in expansion~$(\ref{rozkl1})$ of the process $X(t)$ are independent strictly sub-Gaussian, $\hat X(t)$ is a model of  $X(t)$ defined by $(\ref{model-Xhat})$,    
 $\hat Y(t)$  is defined by $(\ref{model-Yhat})$,  
$D, Q_1, Q_2, S_1, S_2$ are defined in Lemmas~\ref{p6lem-ocen-koef-rozkl1-nestats} and \ref{nas-TN-subg-nestats-rozk1},   
 $\delta >0$, $\varepsilon \in (0;1)$,
 $p \geq 1$, $T>0$. 
%
%
%

Let 
$$
m= \frac{ \varepsilon  \delta^p } 
 {  2^{2p} (p/e)^{p/2} \,  
 T  \sup_{t \in [0,T]}  (\mathsf{E}\exp\{2p X(t)\})^{1/2}     }, 
$$
%
$$
h(t)=t^p (1+ \exp\{ 8 p^2 t^2 \})^{1/4},
\ t \geq 0, 
$$
$x_m$ be the root of the equation 
$$
h(x)=m. 
$$

    If 
\begin{equation}
\label{eq-bound-for-N0}
N_0> \frac{6}{x_m^2}(S_1+S_2 T)^2 + 1,
\end{equation}
\begin{equation}
\label{eq-bound-for-N}
N>\max \left\{ 1 + \log_2 \left(\frac{72(Q_1+Q_2 T)^2}
{5 x_m^2} \right), \,    
1 + \log_8 \left( \frac{18D^2}{7 x_m^2} \right) \right \},
\end{equation}
\begin{equation}
\label{eq-bound-for-Mj}
M_j> 1 + \frac{12}{x_m^2}(Q_1+Q_2 T)^2  \left( 1-\frac{1}{2^{N}} \right)
\quad  (j=0,1,\ldots,N-1),
\end{equation}
then 
the model $\hat Y(t)$ defined by $(\ref{model-Yhat})$ approximates $Y(t)$
with given accuracy~$\delta$ and reliabi\-lity~$1- \varepsilon$ in $L_p([0,T])$.  
\end{thm}

\begin{proof}
We will use the following notations: 
$$
\Delta X(t) = \hat X(t) - X(t), 
$$
$$
\overline{\tau}_X = 
\sup_{t \in [0,T]} \tau(X(t)),
$$
$$ 
\overline{\tau}_{\Delta X} = 
\sup_{t \in [0,T]} \tau(\Delta X(t)),  
$$
$$
c_p=2 (4p/e)^{2p}. 
$$
We will denote the norm in $L_p([0,T])$ as $\|\cdot\|_p$.  

Let us estimate $\mathrm{P} \{\|Y - \hat Y \|_p >\delta \}$. 
We have 
$$
\mathrm{P} \{\|Y - \hat Y \|_p >\delta \} \leq
\frac{\mathsf{E} \|Y - \hat Y \|_p^p}{\delta^p}  
$$
\begin{equation}
\label{oc-ner-Chebysh}
=  \frac{\displaystyle \mathsf{E}\int_{0}^{T} |\exp\{X(t)\}  -  \exp\{\hat X(t)\} |^p dt}
 {\delta^p}. 
\end{equation}

Denote 
$$
\Delta(t)= \mathsf{E}|\exp\{X(t)\}  -  \exp\{\hat X(t)\} |^p. 
$$

An application of Cauchy-Schwarz inequality
yields: 
$$
\Delta(t) = \mathsf{E}  \exp\{p X(t)\} |1- \exp\{\Delta X(t)\}|^{p}
$$
\begin{equation}
\label{ocen-DeltaE}
 \leq 
\left(\mathsf{E} \exp\{2p X(t)\}\right)^{1/2} 
\left(\mathsf{E} |1- \exp\{\Delta X(t)\}|^{2p} \right)^{1/2} . 
\end{equation}

We will need two auxiliary inequalities. 
Using the power mean inequality 
$$
\frac{a+b}{2} \leq 
\left(\frac{a^r + b^r}{2}\right)^{1/r}, 
$$
where $r \geq 1$, 
and setting $a=e^c$, $b=1$  we obtain 
\begin{equation}
\label{ots-sredn-step}
(e^c + 1)^r \leq 2^{r-1} (e^{cr} + 1). 
\end{equation}

It follows from (\ref{razn-exp})
that  
\begin{equation}
\label{oc-s-exp}
|e^a - 1|^q \leq |a|^q (e^a + 1)^q
\end{equation} 
for $q \geq 0$. 



Now let us estimate 
$ \mathsf{E} |1- \exp\{\Delta X(t)\}|^{2p} $, where  $t \in [0,T]$,    using (\ref{oc-s-exp}): 
$$
\mathsf{E} |1- \exp\{\Delta X(t)\}|^{2p} 
\leq  \mathsf{E} |\Delta X(t)|^{2p} (1 + \exp\{\Delta X(t)\})^{2p} 
$$
\begin{equation}
\label{oc-razn-power-2p}
\leq   \left(\mathsf{E} |\Delta X(t)|^{4p}\right)^{1/2}  
\left(\mathsf{E} (1 + \exp\{\Delta X(t)\})^{4p}\right)^{1/2}.   
\end{equation}

Applying (\ref{ots-sredn-step}) we obtain: 
\begin{equation}
\label{ots-E-ot-binoma}
\mathsf{E} (1 + \exp\{\Delta X(t)\})^{4p} 
\leq 
2^{4p-1} \mathsf{E}(\exp\{4p \Delta X(t)\}  + 1).
\end{equation}

It follows from (\ref{ocen-DeltaE}), 
(\ref{oc-razn-power-2p}) and (\ref{ots-E-ot-binoma})   
that for $t \in [0,T]$ 
\begin{equation}
\label{poluitog-ots-Delta}
\Delta(t) \leq 
2^{p-1/4} \left(\mathsf{E}\exp\{2p X(t)\}\right)^{1/2}
\left(\mathsf{E} |\Delta X(t)|^{4p} \right)^{1/4}
(1 + \mathsf{E}\exp\{4p \Delta X(t)\})^{1/4}. 
\end{equation}

Since for $t \in [0,T]$   
$$
\mathsf{E}|\Delta X(t)|^{4p} \leq
c_p \overline{\tau}_{\Delta X}^{4p}
$$
 (see (\ref{ots-mom-subg})) 
and 
$$
\mathsf{E}\exp\{4p \Delta X(t)\} \leq
\exp\{8 p^2 \overline{\tau}_{\Delta X}^2 \}
$$ 
 (see (\ref{oc-expon-mom})) 
we have 
\begin{equation}
\label{oc-Delta-t-final}
\Delta(t) \leq 
2^{p-1/4} c_p^{1/4} \sup_{t \in [0,T]} \left(\mathsf{E}\exp\{2p X(t)\}\right)^{1/2}
h(\overline{\tau}_{\Delta X}), \ t \in [0,T].  
\end{equation}

It follows from 
Lemma~\ref{nas-TN-subg-nestats-rozk1}
and inequalities 
(\ref{eq-bound-for-N0})--(\ref{eq-bound-for-Mj})
 that 
$$
\overline{\tau}_{\Delta X}
= \sup_{t \in [0,T]} (\mathsf{E}(X(t)- \hat X(t))^2)^{1/2}
\leq x_m. 
$$ 
We obtain using (\ref{oc-Delta-t-final})
that 
$$
\Delta(t) \leq  \varepsilon \delta^p/T, \quad t\in [0,T], 
$$
and hence 
\begin{equation}
\label{oc-E-otkl}  
\mathsf{E} \| Y - \hat Y \|_p^p = \int_{0}^{T} \Delta(t) dt 
 \leq  \varepsilon \delta^p .
\end{equation}
Now the statement of the theorem follows from (\ref{oc-ner-Chebysh}) and (\ref{oc-E-otkl}).   
\end{proof}

\begin{ex} {\rm 
  Let us consider a centered Gaussian process  
$X(t)$ with the correlation function 
$$
R(t,s) = \int_{\mathbb{R}} u(t,y) u(s,y)dy,
$$
where 
$$
u(t,y)= \frac{t}{1+ t^2 + \exp\{y^2 \}},
$$
and an arbitrary Battle-Lemari\'e wavelet.  
It is easy to check that the process $Y(t)=\exp\{X(t)\}$ and the Battle-Lemari\'e wavelet satisfy the conditions of Theorem~\ref{thm-T-N-v-Lp}.
}
\end{ex}

 \section*{Acknowledgments} 
  
   The author's research   was supported by  a Swiss Government Excellence  Scholarship.     The author would like to thank  professors Enkelejd Hashorva and Yuriy V. Kozachenko for valuable discussions.

\end{document}